\documentclass{amsart}
\usepackage{amssymb, amsmath, latexsym, mathabx, dsfont,tikz}

\renewcommand{\baselinestretch}{\baselinestretch}
\renewcommand{\baselinestretch}{1.1}
\numberwithin{equation}{section}

\newtheorem{thm}{Theorem}[section]
\newtheorem{lem}[thm]{Lemma}

\newtheorem{prop}[thm]{Proposition}
\newtheorem{conj}[thm]{Conjecture}

\theoremstyle{definition}

\theoremstyle{remark}
\newtheorem{rmk}[thm]{Remark}
\newtheorem{exam}[thm]{Example}

\numberwithin{equation}{section}

\newcommand{\gen}{\text{gen}}

\newcommand{\ord}{\text{ord}}

\newcommand{\z}{{\mathbb Z}}
\newcommand{\q}{{\mathbb Q}}

\begin{document}

\title[The number of representations of squares]{The number of representations of squares by integral ternary quadratic forms}

\author{Kyoungmin Kim and Byeong-Kweon Oh}

\address{Department of Mathematical Sciences, Seoul National University, Seoul 151-747, Korea}
\email{kiny30@snu.ac.kr}

\address{Department of Mathematical Sciences and Research Institute of Mathematics, Seoul National University, Seoul 151-747, Korea}
\email{bkoh@snu.ac.kr}

\thanks{This work  was supported by the National Research Foundation of Korea (NRF-2014R1A1A2056296).}


%%% ----------------------------------------------------------------------

\begin{abstract} Let $f$ be a positive definite integral ternary quadratic form and let $r(k,f)$ be the number of representations of an integer $k$ by $f$. 
In this article we study the number of representations of squares by $f$. 
We say the genus of $f$, denoted by $\gen(f)$, is indistinguishable by squares if  for any integer $n$, $r(n^2,f)=r(n^2,f')$ for any quadratic form $f' \in \gen(f)$.  We find some non trivial genera of ternary quadratic forms which are indistinguishable by squares. We also give some relation between indistinguishable genera by squares and the conjecture given by Cooper and Lam, and we resolve their conjecture completely.       
\end{abstract}

\maketitle

\section{Introduction} 
For a positive definite integral ternary quadratic form 
$$
f(x_1,x_2,x_3)=\sum_{1 \le i, j\le 3}a_{ij}x_ix_j \qquad (a_{ij}=a_{ji} \in \z)  
$$ 
and an integer $n$, we define a set $R(n,f)=\{ (x_1,x_2,x_3) \in \z^3 : f(x_1,x_2,x_3)=n\}$, and 
 $r(n,f)=\vert R(n,f)\vert$. It is well known that $r(n,f)$ is always  finite  if $f$ is positive definite.
Finding a closed formula for $r(n,f)$ or finding all integers $n$ such that $r(n,f) \ne 0$ for an arbitrary ternary quadratic form $f$ are quite old problems which are still widely open. As one of the simplest cases, Gauss showed that if $f$ is a sum of three squares, then $r(n,f)$ is a multiple of the Hurwitz-Kronecker class number. In fact, if the class number of $f$ is one, then Minkowski-Siegel formula gives a closed formula for $r(n,f)$. As a natural modification of the   Minkowski-Siegel formula, it was proved in \cite {k} and \cite {w} that the weighted sum of the representations of quadratic forms in the spinor genus is also equal to the product of local densities except spinor exceptional integers (see also \cite {sp1} for spinor exceptional integers). Hence if the spinor class number $g^{+}(f)$ of $f$ is one, we also have a closed formula for $r(n,f)$. As far as the authors know, there is no known closed formula for $r(n,f)$ except those cases (for some relations between $r(n,f)$'s, see \cite{jlo}). If $r(n,f)=r(n,f')$ for any $f' \in \gen(f)$, it is certain that Minkowski-Siegel formula gives a closed formula for $r(n,f)$. However, Schiemann proved in \cite{s} that for two ternary forms $f$ and $f'$, if $r(n,f)=r(n,f')$ for any  positive integer $n$, then $f$ is isometric to $f'$. 

If we consider a proper subset $S$ of positive integers, then it might be possible that there are non isometric forms $f$ and $f'$ such that $r(n,f)=r(n,f')$ for any integer $n \in S$. In this article, we consider the case when $S$ is the set of perfect squares. We say the genus of a ternary quadratic form $f$ is {\it indistinguishable by squares} if $r(n^2,f)=r(n^2,f')$ for any $f'\in \gen(f)$ and any integer $n$. It is obvious that if the genus of $f$ does not represent any squares of integers, that is, $r(n^2,f')=0$ for any integer $n$ and any $f' \in \gen(f)$, or the class number of $f$ is one, then the genus of $f$ is indistinguishable by squares. If the genus of $f$ is indistinguishable by squares, then Minkowski-Siegel formula gives a closed formula for $r(n^2,f)$ for any integer $n$.    

In 2013, Cooper and Lam gave a conjecture in \cite{cl} on the representations of squares by diagonal ternary quadratic forms representing $1$. In that article, they proved by using some  $q$-series identities, that for the quadratic form $f=x^2+by^2+cz^2$ with $(b,c)=(1,1),(1,2),(1,3),$ $(2,2),(3,3)$, 
$$
r(n^2,f)=\prod_{p\mid 2bc} g(b,c,p,\ord_p(n))\prod_{p \nmid 2bc} h(b,c,p,\ord_p(n)),
$$
where  
$$
h(b,c,p,\ord_p(n))=\frac{p^{\ord_p(n)+1}-1}{p-1}-\left(\frac{-df}p\right)\frac{p^{\ord_p(n)}-1}{p-1}
$$
 and $g(b,c,p,\ord_p(n))$ has to determined on an individual and case-by-case basis, and they conjectured that the above equality also holds for some $64$ pairs of $(b,c)$ (see Table 1 in Section 3).  Recently, Guo, Peng and Qin  in \cite{gpq} verified the conjecture when 
$$
(b,c)=(1,4), (1,5), (1,6), (1,8), (2,3), (2,4), (2,6), (3,6), (4,4), (4,8), (5,5), (6,6)
$$  
by using modular forms of weight $3/2$. 

In this article, we show that the genus of a ternary quadratic form $f$ is indistinguishable by squares if and only if every quadratic form in the genus of $f$ satisfies the condition given in the above conjecture (even if $f$ is non-diagonal, the condition can be extended in an obvious way). Furthermore, we resolve their conjecture completely. Our method is based on Minkowski-Siegel formula on the weighted sum of the representations by quadratic forms in the same genus.

The term lattice will always refer to a positive definite integral $\z$-lattice on an $n$-dimensional positive definite quadratic space over $\q$. Let $L=\z x_1+\z x_2+ \dots+\z x_n$ be a $\z$-lattice of rank $n$. We write
$$
L\simeq (B(x_i,x_j)).
$$
The right hand side matrix is called a {\it matrix presentation} of
$L$.

For two $\z$-lattices $M$ and $L$, a linear map $\sigma : M \to L$ is called a {\it representation} from $M$ to $L$ if it preserves the bilinear form, that is, 
$$
B(\sigma(x),\sigma(y))=B(x,y) \quad \text{for any $x,y \in M$}.
$$  
We define
$$
R(M,L)=\{ \sigma : M \to L \mid \text{$\sigma$ is a representation}\}
$$
and $r(M,L)=\vert R(M,L)\vert$.
In particular, $O(L)=R(L,L)$ which is called the isometry group of $L$, and $o(L)=r(L,L)$.
For any $\z$-lattice $L$, the class containing $L$ in $\gen(L)$ is denoted by $[L]$. As usual, we define 
$$
w(L)=\sum_{[M] \in \gen(L)} \frac1{o(M)}\quad \text{and} \quad   r(K,\gen(L))= \frac 1 {w(L)}  \sum_{[M] \in \gen(L)} \displaystyle \frac{r(K,M)}{o(M)},
$$
for any $\z$-lattice $K$. 
Throughout this article, we assume that every lattice $L$ is positive definite and primitive in the sense that the scale $\mathfrak s(L)$ of $L$ is $\z$, unless stated otherwise.  

Any unexplained notations and terminologies can be found in  \cite {ki} or \cite {om}.

%%%%%%%%%%%%%%%%%%%%%%%%%%%%%%%%%%%%%%%%%%%%%%%%%%%%%%%%%%%%%%%%%%%%%%%%%%%%%%%%%
%%%%%%%%%%%%%%%%%%%%%%%%%%%%%%%%%%%%%%%%%%%%%%%%%%%%%%%%%%%%%%%%%%%%%%%%%%%%%%%%%
\section{representations of squares by ternary quadratic forms}
%%%%%%%%%%%%%%%%%%%%%%%%%%%%%%%%%%%%%%%%%%%%%%%%%%%%%%%%%%%%%%%%%%%%%%%%%%%%%%%%%
%%%%%%%%%%%%%%%%%%%%%%%%%%%%%%%%%%%%%%%%%%%%%%%%%%%%%%%%%%%%%%%%%%%%%%%%%%%%%%%%%

In this section, we investigate various properties of the indistinguishable genus of a ternary quadratic form by squares. 
Let $L$ be a (positive definite integral) $\z$-lattice. As pointed out in the introduction, some genera of ternary $\z$-lattices are obviously indistinguishable by squares. 
   
\begin{lem} \label{space}
 Let $L$ be a ternary $\z$-lattice and let $V=\q\otimes L$ be the quadratic space. Then
$r(n^2,L')=0$ for any $L' \in \gen(L)$ and any  positive integer $n$ if and only if   $d(V_p)=-1$ and  $S_p(V) \ne (-1,-1)_p$  for some prime $p$.
\end{lem}

\begin{proof} The lemma follows directly from the fact that
$r(n^2,L')=0$ for any $L' \in \gen(L)$ and any  positive integer $n$ if and only if $1$ is not represented by $V$. 
\end{proof}

%From now on, we always assume that $V=\q\otimes L$ represents $1$ and the class number $h(L)$ of $L$ is bigger than $1$. 

\begin{rmk} {\rm
In fact, it is possible that $1$ is not represented by a genus which is indistinguishable by squares. For example, 
let $L\simeq\langle2,3,24\rangle$. Then one may easily check that the class number of $L$ is $2$ and the other lattice in the genus is $L'\simeq\begin{pmatrix} 5&1 \\ 1&5\end{pmatrix}\perp \langle6\rangle$. By checking the local structure at $p=2$ and  $3$, one may easily show that $r(n^2,L)=r(n^2,L')=0$ for any integer $n$ not divisible by $6$. Assume that 
$$
36n^2=2x^2+3y^2+24z^2
$$ 
for some integers $n$ and $x,y,z$. Then one may easily check that $x \equiv 0 \pmod 6$ and $y\equiv 0 \pmod 2$. Therefore we have  $r(36n^2,L)=r(3n^2,\langle1,2,6\rangle)$. Similarly, one may also check that  $r(36n^2,L')=r(3n^2,\langle1,2,6\rangle)$. Therefore we have $r(n^2,L)=r(n^2,L')$ for any integer $n$. } 
\end{rmk} 

\begin{lem} \label{important-l}
Let $L$ be a ternary $\z$-lattice and let $n$ be a positive integer. For any prime $p$, we assume that $\ord_p(n)=\lambda_p$. If $1$ is represented by the genus of $L$, then we have 
$$
\begin{array}  {rl}
\displaystyle \frac{r(n^2,\text{gen}(L))}{r(1,\gen(L))}
&=n\displaystyle  \prod_{p\vert n,  p\vert 2dL} \frac{\alpha_{p}(n^2,L)}{\alpha_{p}(1,L)}\displaystyle  \prod_{ p\vert n, p\nmid 2dL} \frac{\alpha_{p}(n^2,L)}{\alpha_{p}(1,L)}\\
&=\displaystyle \prod_{ p\vert n, p\vert 2dL}\!\!\! p^{\lambda_p} \cdot \frac{\alpha_{p}(n^2,L)}{\alpha_{p}(1,L)}
\displaystyle \prod_{ p\vert n, p\nmid 2dL} \!\!\! \left(\frac{p^{\lambda_p +1}-1}{p-1}-\left(\frac{-dL}{p}\right)\frac{p^{\lambda_p}-1}{p-1}\right).\\
\end{array}
$$
In particular, if the lattice $L$ has class number 1, then we have
$$
r(n^2,L)=r(1,L)\prod_{ \substack  p\vert n, \\p\vert 2dL}p^{\lambda_p}\cdot \frac{\alpha_{p}(n^2,L)}{\alpha_{p}(1,L)}
\prod_{ \substack  p\vert n, \\p\nmid 2dL} \left(\frac{p^{\lambda_p +1}-1}{p-1}-\left(\frac{-dL}{p}\right)\frac{p^{\lambda_p}-1}{p-1}\right).
$$
\end{lem}

\begin{proof}
The Minkowski-Siegel formula implies
$$
r(n^2,\text{gen}(L))=\pi^{\frac{3}{2}}\Gamma \left(\frac{1}{2}\right)^{-1}\cdot\frac{1}{\sqrt{dL}}\cdot n \prod_{p<\infty}\alpha_p({n^2,L_p}),
$$
where $\alpha_p$ is the local density. If $p$ does not divide $2dL$, then   by \cite {y}, 
\begin{displaymath}
\alpha_p(n^2,L)=\alpha_p(p^{2\lambda_p},L)=1+\frac{1}{p}-\frac{1}{p^{\lambda_p+1}}+\left( \frac{-dL}{p}\right) \frac{1}{p^{\lambda_p+1}}.
\end{displaymath}
The lemma follows from this. \end{proof}

\begin{rmk} {\rm
When the class number of a ternary lattice $L$ is $1$, the above lemma gives a closed formula, which is, in principle, a finite product of  local densities, for the number of representations of squares by $L$. This could be extended to higher rank cases. Let $L$ be a $\z$-lattice with $h(L)=1$. 
 If the rank of the $\z$-lattice $L$ is an odd (even) integer greater than $1$ ($2$), then we might have a closed formula of $r(n^2,L)$ ($r(n,L)$, respectively) which is  essentially given by a finite product of local densities. For example, if the rank of $L$ is $4$, then we have
$$           
r(n,L)=r(1,L)\prod_{p\mid n, p\mid 2dL} p^{\ord_p(n)}\frac{\alpha_p(n,L)}{\alpha_p(1,L)} \prod_{p\mid n, p\nmid 2dL}\frac {p^{\ord_p(n)+1}-\left(\frac{dL}p\right)^{\ord_p(n)+1}}{p-\left(\frac{dL}p\right)}.
$$
There are some articles dealing with this subject by using different methods such as $q$-series or modular forms. For example, see  \cite{aw}, \cite{c}, \cite{cly} and \cite{esy}.}
\end{rmk}

\begin{lem}\label{hecke-act}
Let $L$ be a ternary $\z$-lattice. If for any $L' \in \gen(L)$, $r(n^2,L)=r(n^2,L')$ for any integer $n$ such that every prime factor of $n$ divides $2dL$, then the genus of $L$ is indistinguishable by squares.   
\end{lem}

\begin{proof}
The action of Hecke operators $T(p^2)$ for any prime $p\nmid 2dL$ on theta series of  the $\z$-lattice $L$ gives
\begin{equation} 
r(p^2n,L)+\left(\frac{-ndL}{p}\right)r(n,L)+p \cdot r\left(\frac{n}{p^2},L\right)=\!\!\!\sum_{[L']\in \gen(L)} \frac{r^*(pL',L)}{o(L')}r(n,L').
\end{equation}
Here, if $p^2\nmid n$, then $r\left(\frac{n}{p^2},L\right)=0$ and 
\begin{displaymath}
r^*(pL',L) = \left\{ \begin{array}{ll}
r(pL',L)-o(L) & \textrm{if $L\simeq L'$},\\
r(pL',L) & \textrm{otherwise}.
\end{array} \right.\end{displaymath} 
For details, see \cite {an}. It is well known that
\begin{equation}
\sum_{[L']\in \gen(L)} \frac{r^*(pL',L)}{o(L')}=p+1.
\end{equation}
Assume that $r(n^2,L)=r(n^2,L')$ for any $L' \in \gen(L)$. Then by (2.2), we have 
$$
r(p^2n^2,K)=\left(p+1-\left(\frac{-n^2dL}p\right)\right)r(n^2,K)-p \cdot r\left(\frac{n^2}{p^2},K\right),
$$ 
for any $\z$-lattice $K \in \gen(L)$. 
Therefore if $n$ is not divisible by $p$, then $r(p^2n^2,L)=r(p^2n^2,L')$ for any $L' \in \gen(L)$. The lemma follows from induction on the number of prime factors not dividing $2dL$ counting multiplicity. 
\end{proof}

Now we collect some known results on the number of representations of integers by ternary quadratic forms, which are needed later.  Let $L$ be a ternary $\z$-lattice. For any prime $p$, the $\lambda_p$-transformation (or Watson transformation) is defined as follows: 
$$
\Lambda_p(L)= \{ x \in L : Q(x + z) \equiv Q(z) \  (\text{mod} \ p) \mbox{ for
all $z \in L$}\}.
$$
 Let $\lambda_p(L)$ be the primitive
lattice obtained from $\Lambda_p(L)$ by scaling $V=L\otimes \mathbb Q$ by a suitable
rational number.

\begin{lem} \label{aniso}
Let $L$ be a ternary $\z$-lattice and let $p$ be an odd prime.
If the unimodular component in a Jordan decomposition of $L_p$ is anisotropic, then
$$
r(pn,L)=r(pn,\Lambda_p(L)).
$$
\end{lem}

\begin{proof} See \cite {co}. 
\end{proof}

Now assume that the unimodular component in a Jordan decomposition of $L_p$ is isotropic. Assume that $p$ is an odd prime dividing $dL$.
Then by Weak Approximation Theorem, there exists a basis $\{x_1, x_2, x_3 \}$ for $L$ such that
$$
(B(x_i,x_j))\equiv\begin{pmatrix}0&1\\ 1&0\end{pmatrix}\perp \langle p^{\ord_p(dL)} \delta\rangle \ (\text{mod} \ p^{\ord_p(dL)+1}),
$$
where $\delta$ is an integer not divisible by $p$. We define
$$
\Gamma_{p,1}(L) = \z px_1 + \z x_2+ \z x_3, \qquad \Gamma_{p,2}(L) = \z x_1 + \z px_2+ \z x_3.
$$  
Note that $\Gamma_{p,1}(L)$ and $\Gamma_{p,2}(L)$ are unique sublattices of $L$ with index $p$ whose scale is divisible by $p$. For some properties of these sublattices of $L$, see \cite {jlo}. 

\begin{lem} \label{iso}
Under the same assumptions given above,  we have 
$$
r(pn,L)=r(pn,\Gamma_{p,1}(L))+r(pn,\Gamma_{p,2}(L))-r(pn,\Lambda_p(L)).
$$ 
\end{lem}

\begin{proof}
See Proposition 4.1 of \cite{jlo}.
\end{proof}

%%%%%%%%%%%%%%%%%%%%%%%%%%%%%%%%%%%%%%%%%%%%%%%%%%%%%%%%%%%%%%%%%%%%%%%%%%%%%%%%%%
%%%%%%%%%%%%%%%%%%%%%%%%%%%%%%%%%%%%%%%%%%%%%%%%%%%%%%%%%%%%%%%%%%%%%%%%%%%%%%%%%
\section{Factors of the number of representations of squares}
%%%%%%%%%%%%%%%%%%%%%%%%%%%%%%%%%%%%%%%%%%%%%%%%%%%%%%%%%%%%%%%%%%%%%%%%%%%%%%%%%
%%%%%%%%%%%%%%%%%%%%%%%%%%%%%%%%%%%%%%%%%%%%%%%%%%%%%%%%%%%%%%%%%%%%%%%%%%%%%%%%%

Let $L$ be a ternary $\z$-lattice whose genus is indistinguishable by squares. Then Lemma \ref{important-l} gives a closed formula on $r(n^2,L)$. In this section, we resolve the conjecture given by Cooper and Lam in \cite {cl} by using this observation.     

\begin{table}[t]
\caption{Data for Conjecture 3.1}
\centering
\begin{tabular}{c l}
\hline
$b$ &$\quad c$ \\
\hline
1 &\quad 1, 2, 3, 4, 5, 6, 8, 9, 12, 21, 24\\
2 &\quad 2, 3, 4, 5, 6, 8, 10, 13, 16, 22, 40, 70 \\
3 &\quad 3, 4, 5, 6, 9, 10, 12, 18, 21, 30, 45\\
4 &\quad 4, 6, 8, 12, 24\\
5 &\quad 5, 8, 10, 13, 25, 40 \\
6 &\quad 6, 9, 16, 18, 24\\
8 &\quad 8, 10, 13, 16, 40\\
9 &\quad 9, 12, 21, 24\\
10 &\quad 30\\
12 &\quad 12\\
16 &\quad 24\\
21 &\quad 21\\
24 &\quad 24\\
\hline
\end{tabular}
\end{table}

\begin{conj} \label{conj}
Let $b$ and $c$ be any integers given in Table 1. Let $n$ be a positive integer and let $\lambda_p=\ord_p(n)$ for any prime  $p$.  
Then the number $r(n^2,\langle 1,b,c\rangle)$ of the integer solutions of the diophantine equation
$$
n^2=x^2+by^2+cz^2
$$ 
is given by the formula of the type
$$
r(n^2,\langle1,b,c\rangle)=\left( \prod_{p \mid 2bc}g(b,c,p,\lambda_p)\right)\left( \prod_{p \nmid 2bc}h(b,c,p,\lambda_p)\right),
$$
where 
\begin{equation}
h(b,c,p,\lambda_p)=\frac{p^{\lambda_p +1}-1}{p-1}-\left(\frac{-bc}{p}\right)\frac{p^{\lambda_p}-1}{p-1}
\end{equation}
and $g(b,c,p,\lambda_p)$ has to be determined on an individual and case-by-case basis.
\end{conj}

Cooper and Lam verified the conjecture when 
$$
(b,c)=(1,1), (1,2), (1,3), (2,2), (3,3) 
$$
by using some identities on $q$-series. In 2014, Guo, Peng and Qin  in \cite{gpq} verified the conjecture when 
$$
(b,c)=(1,4), (1,5), (1,6), (1,8), (2,3), (2,4), (2,6), (3,6), (4,4), (4,8), (5,5), (6,6)
$$  
by using modular forms of weight $3/2$. 

In fact, by Lemma \ref{important-l}, the conjecture holds when the class number of $\ell_{b,c}=\langle 1,b,c\rangle$ is one. In this case, we have 
$$
\prod_{p\mid 2bc}g(b,c,p,\lambda_p)=r(1,\ell_{b,c})\prod_{p\mid 2bc}p^{\lambda_p}\cdot \frac{\alpha_{p}(n^2,\ell_{b,c})}{\alpha_{p}(1,\ell_{b,c})}.
$$
Therefore we may assume that the class number of $\ell_{b,c}$ is greater than $1$, that is,
$$
(b,c)=(2,13), (2,22), (2,40), (2,70), (3,5), (3,21), (3,45), (5,13), (8,10) \ \ \text{and} \ \ (8,13).
$$
In fact, one may easily verify that $h(\ell_{b,c})=2$ in all cases given  above. In these cases, we define the other $\z$-lattice in the genus of $\ell_{b,c}$ by $m_{b,c}$.

Let $L$ be any ternary $\z$-lattice which is not necessarily diagonal.
We say $L$ {\it satisfies the condition in Conjecture \ref{conj}} if $r(n^2,L)$ has the same closed formula in the conjecture.  
Note that $bc$ in (3.1) should be replaced by $dL$ in general case. 

Suppose that $L$ satisfies the condition in Conjecture \ref{conj}. 
For any integer $n$, let $P(n)$ be the set of prime factors of $n$. 
Let $n=n_1n_2$, where $P(n_1)\subset P(2dL)$ and $(n_2,2dL)=1$. Define $\lambda_p=\ord_p(n)$ for any prime $p$.  Since
$$
r(n_1^2n_2^2,L)=\prod_{p\mid2dL}g_p(L_p,\lambda_p)\prod_{p\nmid2dL}h_p(L_p,\lambda_p),
\quad r(n_1^2,L)=\prod_{p\mid2dL}g_p(L_p,\lambda_p),
$$
where 
$$
h_p(L_p,\lambda_p)=\frac{p^{\lambda_p +1}-1}{p-1}-\left(\frac{-dL}{p}\right)\frac{p^{\lambda_p}-1}{p-1},
$$
we have
\begin{equation}
r(n_1^2n_2^2,L)=r(n_1^2,L)\prod_{p\nmid2dL}h_p(L_p,\lambda_p).
\end{equation}
Conversely, if $L$ satisfies the condition (3.2), then $L$ satisfies the condition in Conjecture \ref{conj}.

\begin{thm}\label{interesting} Let $L$ be a ternary $\z$-lattice.    
Every $\z$-lattice in the genus of $L$ satisfies the condition in Conjecture \ref{conj} if and only if $\gen(L)$ is indistinguishable by squares.
\end{thm}

\begin{proof}
Suppose that $r(n^2,L)=r(n^2,L')$ for any integer $n$ and any $\z$-lattice $L'$ in the genus of $L$. Let $n=n_1n_2$, where $P(n_1)\subset P(2dL)$ and $(n_2,2dL)=1$. First assume that $r(n_1^2,L)\ne 0$. By Minkowski-Siegel formula, we have
$$
\frac{r(n_1^2n_2^2,L')}{r(n_1^2,L')}=\frac{r(n_1^2n_2^2,\gen(L))}{r(n_1^2,\gen(L))}=\!\!
\prod_{p\mid 2dL}\frac{\alpha_p(n_1^2n_2^2,L_p')}{\alpha_p(n_1^2,L_p')}\!\!\prod_{p\nmid 2dL}h_p(L_p,\lambda_p)=\!\!\!\prod_{p\nmid 2dL}h_p(L_p,\lambda_p),
$$
for any $\z$-lattice $L' \in \gen(L)$. Hence we have 
$$
r(n_1^2n_2^2,L')=r(n_1^2,L')\prod_{p\nmid2dL}h(L_p,\lambda_p).
$$
Now assume that $r(n_1^2,L)=0$. Then by Lemma \ref{hecke-act}, $r(n_1^2n_2^2,L')=0$ for any $L'\in \gen(L)$. 
Therefore $L'$ satisfies the condition in Conjecture \ref{conj} for any $L'\in \gen(L)$.

Conversely, suppose that every $\z$-lattice in the genus of $L$ satisfies the condition in Conjecture \ref{conj}. Let 
$\gen(L)=\{[L]=[L_1], [L_2],\dots, [L_h]\}$. Let $n_1$ be any integer such that $P(n_1) \subset P(2dL)$.
 Note that for any prime $p\nmid 2dL$ and any integer $i\in \{1,2,\dots, h\}$,
$$
r(p^2n_1^2,L_i)+\left(\frac{-dL}{p}\right)r(n_1^2,L_i)=\sum_{[L']\in \gen(L)} \frac{r^*(pL',L_i)}{o(L')}r(n_1^2,L').
$$ 
Since $r(p^2n_1^2,L_i)=\left(p+1-\left(\frac{-dL}{p}\right)\right)r(n_1^2,L_i)$ by the assumption, we have 
\begin{displaymath}
\pi_p(L)\left(\begin{array}{c} r(n_1^2,L_1) \\ \vdots\\r(n_1^2,L_h)\end{array} \right)=(p+1)\left(\begin{array}{c} r(n_1^2,L_1) \\ \vdots\\r(n_1^2,L_h)\end{array} \right),
\end{displaymath}
where $\pi_p(L)=\left(\frac{r^*(pL_j,L_i)}{o(L_j)}\right)$ is the transpose of the Eichler's Anzahlmatrix of $L$ at $p$ (see \cite{jlo}).
This implies that $\pi_p(L)$ has an eigenvalue $p+1$ corresponding to the eigenvector $(r(n_1^2,L_1),\dots, r(n_1^2,L_h))$. Assume that 
$$
r(n_1^2,L_k)=\max (r(n_1^2,L_1),r(n_1^2,L_2),\dots,r(n_1^2,L_h)).
$$ 
Then
$$
\begin{array} {rl}
(p+1)r(n_1^2,L_k)&\!\!\!=\displaystyle \sum_{i=1}^{h}\frac{r^*(pL_i,L_k)}{o(L_i)}r(n_1^2,L_i)\\
&\hskip 5pc \le\displaystyle\sum_{i=1}^{h}\frac{r^*(pL_i,L_k)}{o(L_i)}r(n_1^2,L_k)=(p+1)r(n_1^2,L_k).
\end{array}
$$
This implies that
$$
\frac{r^*(pL_i,L_k)}{o(L_i)}r(n_1^2,L_i)=\frac{r^*(pL_i,L_k)}{o(L_i)}r(n_1^2,L_k).
$$
Now by class linkage Theorem proved in \cite {hjs}, for any integer $i=1,2,\dots,h$,  there is a prime $p\nmid 2dL$ such that $r^*(pL_i,L_k)\ne 0$. This implies that $r(n_1^2,L_i)=r(n_1^2,L_k)$. The theorem follows from this by Lemma \ref{hecke-act}.  \end{proof}

\begin{rmk} 
{\rm Assume that the class number of a ternary $\z$-lattice $L$ is two. Then one may easily show that if $L$ satisfies the condition in Conjecture \ref{conj}, then so is the other $\z$-lattice in the genus of $L$.}
\end{rmk} 

From now on, we prove the conjecture when the class number of the $\z$-lattice $\ell_{b,c}$ is two. In fact, we will show that
each genus of $\ell_{b,c}$ is indistinguishable by squares except the cases when $(b,c)=(3,5), (3,21)$ and $(3,45)$. In the exceptional cases, we show that $r(n^2,\ell_{b,c})= r(n^2,m_{b,c})$ only when $n$ is odd.

%%%%%%%%%%%%%%%%%%%%%%%%%%%%%%%%%%%%%%%%%%%%%%%%%%%%%%%%%%%%%%%%%%%%%%%%%%%%%%%%%%%
%%%%%%%%%%%%%%%%%%%%%%%%%%%%%%%%%%%%%%%%%%%%%%%%%%%%%%%%%%%%%%%%%%%%%%%%%%%%%%%%%%%

\begin{thm} \label{easy-1} The genera  $\gen(\ell_{2,40})$,  $\gen(\ell_{2,22})$ and  
 $\gen(\ell_{2,70})$ are all indistinguishable by squares.  
\end{thm}

\begin{proof} Note that 
$$
\gen(\ell_{2,40})=\{\ell_{2,40},\ \ell_{8,10}\}, \qquad  \gen(\ell_{2,22})=\left\lbrace\ell_{2,22},   \    
m_{2,22}=\langle1\rangle\perp\begin{pmatrix} 6&2\\ 2&8\end{pmatrix}\right\rbrace
$$ 
and 
$$ 
 \gen(\ell_{2,70})=\left\lbrace\ell_{2,70},
  \  m_{2,70}=\langle1\rangle\perp\begin{pmatrix} 8&2\\ 2&18\end{pmatrix}\right\rbrace.
$$
Let $n$ be any integer such that $n \equiv 0,1 \pmod 4$. Define a map
$$
\begin{cases}  f_n: R(n,\ell_{8,10}) \mapsto R(n,\ell_{2,40}) \quad &\text{by} \quad f_n(x,y,z)=\left(x,2y,\displaystyle\frac{z}{2}\right), \\
f_n: R(n,\ell_{2,22}) \mapsto R\left(n,m_{2,22}\right) \quad &\text{by} \quad f_n(x,y,z)=\left(x,2z,\displaystyle\frac{y-z}{2}\right),\\ 
f_n: R(n,\ell_{2,70}) \mapsto R\left(n,m_{2,70}\right) \quad &\text{by} \quad f_n(x,y,z)=\left(x,\displaystyle \frac{y-z}{2},2z\right).\\ 
\end{cases}
$$
One may easily check that $f_n$ is a well defined bijective map.  The lemma follows directly from this. 
\end{proof} 

%%%%%%%%%%%%%%%%%%%%%%%%%%%%%%%%%%%%%%%%%%%%%%%%%%%%%%%%%%%%%%%%%%%%%%%%%%%%%%%%%
%%%%%%%%%%%%%%%%%%%%%%%%%%%%%%%%%%%%%%%%%%%%%%%%%%%%%%%%%%%%%%%%%%%%%%%%%%%%%%%%%

\begin{lem} \label{odd}
For any positive integer $n$ such that $n \equiv 1 \pmod 8$, 
$$
r(n,\ell_{3,5})=r\left(n, m_{3,5}\right), \ \   r(n,\ell_{3,21})=r\left(n, m_{3,21}=\langle1\rangle\perp\begin{pmatrix} 6&3\\ 3&12\end{pmatrix}\right)
$$ 
and 
$$
r(n,\ell_{3,45})=r\left(n, m_{3,45}=\langle1\rangle\perp\begin{pmatrix} 12&3\\ 3&12\end{pmatrix}\right).
$$
\end{lem}

\begin{proof}
Since proofs are quite similar to each other, we only provide the proof of the first case. Note that $m_{3,5}=\langle1\rangle\perp\begin{pmatrix} 2&1\\ 1&8\end{pmatrix}$. Let $n$ be an integer such that $n \equiv 1 \pmod 8$.  Define
$$
A_1(n)=\{(x,y,z)\in \z^3\mid x^2+3y^2+5z^2=n, \ x\equiv y\equiv z\equiv 1\!\! \!\!\!\pmod 2\},
$$
$$
A_2(n)=\{(x,y,z)\in \z^3\mid x^2+3y^2+5z^2=n,  \ x\equiv 1\!\!\!\!\! \pmod 2, \ y\equiv z \equiv 0\!\! \!\!\!\pmod 2\},
$$
and
$$
A_3(n)=\{(x,y,z)\in \z^3\mid x^2+3y^2+5z^2=n, \ x\equiv y \equiv 0\!\!\!\!\! \pmod 2, \ z \equiv 1\!\! \!\!\!\pmod 2\}.
$$
 We also define 
$$
B_1(n)=\{(x,y,z)\in \z^3\mid x^2+2y^2+8z^2+2yz=n,  \ x\equiv y \equiv z \equiv 1 \!\! \!\!\!\pmod 2\},
$$
$$
B_2(n)=\{(x,y,z)\in \z^3\mid x^2+2y^2+8z^2+2yz=n, x\equiv 1\!\!\!\!\! \pmod 2, \ y\equiv z \equiv 0\!\! \!\!\!\pmod 2\}, 
$$
and
$$
B_3(n)=\{(x,y,z)\in \z^3\mid x^2+2y^2+8z^2+2yz=n,\ x\equiv z\equiv 1 \!\!\!\!\!\pmod 2, \ y\equiv 0 \!\!\!\!\!\pmod 4\}.
$$
Then it is clear that 
\begin{displaymath}
R(n,\ell_{3,5})=A_1(n) \cup A_2(n) \cup A_3(n) \ \ \text{and} \  \ 
R(n,m_{3,5})=B_1(n) \cup B_2(n) \cup B_3(n).
\end{displaymath}
One may easily check that $y\equiv z \pmod 4$ for any $(x,y,z) \in A_2(n)$, $x\not \equiv y \pmod 4$ for any $(x,y,z) \in A_3(n)$ and $y \not\equiv z\pmod 4$ for any $(x,y,z) \in B_1(n)$. 

We  prove that 
$$
\vert A_1(n)\vert+\vert A_2(n)\vert=\vert B_1(n)\vert+\vert B_2(n)\vert\quad \text{and}\quad  \vert A_3(n)\vert=\vert B_3(n)\vert,
$$ 
which implies the assertion. 
First, we define a map $f:A_1(n)\cup A_2(n)\mapsto B_1(n)\cup B_2(n)$ by
$$
f(x,y,z)=\left(x, \frac{y+3z}{2}, \frac{y-z}{2}\right).
$$
Since
$$
x^2+3y^2+5z^2=x^2+2\left(\frac{y+3z}2\right)^2+2\left(\frac{y+3z}2\right)\left(\frac{y-z}2\right)+8\left(\frac{y-z}2\right)^2
$$
and 
$$
\displaystyle \frac{y+3z}2-\displaystyle\frac{y-z}2 =2z\equiv 0 \pmod 2,
$$
the above map $f$ is well defined. Furthermore one may easily check that $f$ is bijective.  

To show that $\vert A_3(n)\vert=\vert B_3(n)\vert$, we define 
$$
A^0_3(n)=\{ (x,y,z) \in A_3(n) \mid x-y-2z \equiv 4 \!\!\!\!\!\pmod 8\}
$$   
and
$$
B^0_3(n)=\{ (x,y,z) \in B_3(n) \mid 2x-y+2z \equiv 0 \!\!\!\!\!\pmod 8\}.
$$
Since $(x,y,z)\in A_3(n) \iff (x,y,-z) \in A_3(n)$, we have $2\vert A^0_3(n)\vert=\vert A_3(n)\vert$. Similarly, we also have 
 $2\vert B^0_3(n)\vert=\vert B_3(n)\vert$ from the fact that  $(x,y,z)\in B_3(n) \iff (-x,y,z) \in B_3(n)$.
Now we define a map $g:A^0_3(n)\mapsto B^0_3(n)$ by 
$$
g(x,y,z)=\left(\frac{x+3y}{2}, \frac{x-y+2z}{2}, \frac{-x+y+2z}{4}\right)
$$
and a map $h:B^0_3(n)\mapsto A^0_3(n)$ by
$$
h(x,y,z)=\left(\frac{2x+3y-6z}{4}, \frac{2x-y+2z}{4}, \frac{y+2z}{2}\right).
$$
One may easily check that both $g$ and $h$ are well defined and $h\circ g=\text{id},\ g\circ h=\text{id}$. 
\end{proof}

%%%%%%%%%%%%%%%%%%%%%%%%%%%%%%%%%%%%%%%%%%%%%%%%%%%%%%%%%%%%%%%%%%%%%%%%%%%%%%%%%%%%%%%%%%%%%%%%%%%%%%%
%%%%%%%%%%%%%%%%%%%%%%%%%%%%%%%%%%%%%%%%%%%%%%%%%%%%%%%%%%%%%%%%%%%%%%%%%%%%%%%%%%%%%%%%%%%%%%%%%%%%%%%
%%%%%%%%%%%%%%%%%%%%%%%%%%%%%%%%%%%%%%%%%%%%%%%%%%%%%%%%%%%%%%%%%%%%%%%%%%%%%%%%%%%%%%%%%%%%%%%%%%%%%%%

We put 
$$
 K_1= \begin{pmatrix} 1&0&0\\ 0&2&1\\ 0&1&7\end{pmatrix}, \ \  K_2=\begin{pmatrix} 2&0&1\\ 0&3&1\\ 1&1&3\end{pmatrix}, \ \
L_1= \begin{pmatrix} 2&0&1\\ 0&4&2\\ 1&2&8\end{pmatrix},\ \  L_2=\begin{pmatrix} 2&1&1\\ 1&4&0\\ 1&0&8\end{pmatrix}, 
$$
 and
$$
T= \begin{pmatrix} 2&0&1\\ 0&4&1\\ 1&1&4\end{pmatrix}.
$$

\begin{thm}  Both $\gen(\ell_{2,13})$ and $\gen(\ell_{8,13})$ are indistinguishable by squares. 
\end{thm}

\begin{proof} Note that 
$$
 m_{2,13}= \langle1\rangle \perp \begin{pmatrix} 5&2\\ 2&6\end{pmatrix}, \ \ m_{8,13}= \langle1\rangle \perp \begin{pmatrix} 5&1\\ 1&21\end{pmatrix}.
$$
For any integer $n$,
$$
r(4n^2,\ell_{8,13})=r(n^2,\ell_{2,13}) \quad \text{and} \quad   r(4n^2, m_{8,13})=r(n^2,m_{2,13}).
$$
If $n$ is odd, then 
$$
r(n^2,\ell_{2,13})=r(n^2,\ell_{8,13}) \quad \text{and} \quad   r(n^2, m_{2,13})=r(n^2,m_{8,13}).
$$
Hence $\gen(\ell_{2,13})$ is indistinguishable by squares if and only if $\gen(\ell_{8,13})$ is indistinguishable by squares. Therefore it suffices to show that $\gen(\ell_{2,13})$\ is indistinguishable by squares.  Note that 
$$
r(4n^2,\ell_{2,13})=r(2n^2, K_1)=r(2n^2,L_1) \ \text{and}\ r(4n^2, m_{2,13})=r(2n^2, K_2)=r(2n^2,L_2).
$$
Now by Lemma \ref{iso}, we have 
$$
r(8n^2,L_1)=2r(4n^2,T)-r(2n^2,K_2) \quad \text{and} \quad r(8n^2,L_2)=2r(4n^2,T)-r(2n^2,K_1).
$$
By combining the last two equalities, we have
$$
r(16n^2,\ell_{2,13})-r(4n^2,\ell_{2,13})=r(16n^2,m_{2,13})-r(4n^2,m_{2,13}).
$$
Furthermore, since $r(4n^2,\ell_{2,13})=r(n^2,\ell_{2,13})$ and $r(4n^2,m_{2,13})=r(n^2,m_{2,13})$ for any odd integer $n$, it suffices to show that 
$r(n^2,\ell_{2,13})=r(n^2,m_{2,13})$ for any odd integer $n$. Now by Lemma \ref{aniso}, we know that  
$r(13^2\cdot n^2,\ell_{2,13})=r(n^2,\ell_{2,13})$ and   $r(13^2\cdot n^2,m_{2,13})=r(n^2,m_{2,13})$. Therefore the theorem follows directly from Lemma \ref{hecke-act}. 
\end{proof}

\begin{prop} The genus of $\ell_{5,13}$ is indistinguishable by squares. 
\end{prop}
\begin{proof} 
Note that $\gen(\ell_{5,13})=\left\lbrace\ell_{5,13}, m_{5,13}=\langle1\rangle \perp  \begin{pmatrix} 2&1\\ 1&33 \end{pmatrix}\right\rbrace$. Since $\ell_{5,13}$ is aniso-tropic over $\z_5$ and $\z_{13}$, $r(n^2,\ell_{5,13})=r(n^2, m_{5,13})$ for any odd integer $n$ by Lemma \ref{hecke-act}. If $x^2+5y^2+13z^2=4n^2$, then $x,y,z$ are all even. Therefore  
$$
r(4n^2,\ell_{5,13})=r(n^2,\ell_{5,13}) \quad \text{and}\quad r(4n^2,m_{5,13})=r(n^2,m_{5,13}).
$$
The proposition follows from this. 
\end{proof}

\begin{rmk} {\rm All genera containing $\ell_{3,5}$, $\ell_{3,21}$ or $\ell_{3,45}$ are not indistinguishable by squares. 
 For example, $r(4,\ell_{3,k})=6 \ne 2= r(4,m_{3,k})$ for any $k=5,21$ and $45$. However by Lemma \ref{odd}, $r(n^2,\ell_{3,k})=r(n^2,m_{3,k})$ for any odd integer $n$.}    
\end{rmk}

\begin{exam} As pointed out earlier, if the genus of a $\z$-lattice $L$ is indistinguishable by squares, then we may have a closed formula for $r(n^2,L')$ for any $\z$-lattice $L' \in \gen(L)$. 
For example, one may easily show that
$$
\frac{\alpha_2(n^2,\ell_{2,13})}{\alpha_2(1,\ell_{2,13})}=\frac{2^{\max(1,\lambda_{2})+1}-3}{2^{\lambda_2}} \quad \text{and} \quad \frac{\alpha_{13}(n^2,\ell_{2,13})}{\alpha_{13}(1,\ell_{2,13})}=\frac{1}{13^{\lambda_{13}}},
$$
where $\lambda_p=\ord_p(n)$ for any prime $p$. Therefore by Lemma \ref{important-l}, we have 
$$
r(n^2,\ell_{2,13})\!=r(n^2,m_{2,13})\!=
2(2^{\max(1,\lambda_{2})+1}-3)\!\!\!\!\!\prod_{\substack {p\mid n, \\ (p,26)=1}}\!\!\!\!\!\left( \frac{p^{\lambda_p +1}-1}{p-1}-\left(\frac{-26}{p}\right)\frac{p^{\lambda_p}-1}{p-1}\right).
$$
\end{exam}

\end{document}